\documentclass[10pt,psamsfonts]{amsart}
\usepackage{amsmath}
\usepackage{amsthm}
\usepackage{amssymb}
\usepackage{amscd}
\usepackage{amsfonts}
\usepackage{amsbsy}
\usepackage{graphicx}
\usepackage[dvips]{psfrag}
\usepackage{array}
\usepackage{color}
\usepackage{epsfig}
\usepackage{url}
\usepackage{overpic}
\usepackage{tikz-cd}
\usepackage{enumitem}
\usepackage{hyperref}
\usepackage{soul}
\usepackage{float}
\usepackage[normalem]{ulem}

\usepackage{textcomp}

\newcommand{\R}{\ensuremath{\mathbb{R}}}
\newcommand{\N}{\ensuremath{\mathbb{N}}}

\newcommand{\bx}{{\bf x}}
\newcommand{\by}{{\bf y}}
\newcommand{\bz}{{\bf z}}

\newcommand{\C}{\ensuremath{\mathbb{C}}}

\newtheorem {theorem} {Theorem}
\newtheorem {definition} {Definition}

\newtheorem {corollary}{Corollary}
\newtheorem {lemma}{Lemma}

\newtheorem {mtheorem} {Theorem}

\usepackage{mathpazo}
\usepackage[foot]{amsaddr}

\begin{document}
\renewcommand{\arraystretch}{1.5}

\title[Periodic and antiperiodic aspect of the Floquet normal form]{On the periodic and antiperiodic aspects\\ of the Floquet normal form}
\author[D.D. Novaes and P.C.C.R. Pereira]
{Douglas D. Novaes$^1$ and Pedro C.C.R. Pereira$^2$}

\address{Departamento de Matem\'{a}tica - Instituto de Matemática, Estatística e Computação Científica (IMECC) - Universidade Estadual de Campinas (UNICAMP), Rua S\'{e}rgio Buarque de Holanda, 651, Cidade
Universit\'{a}ria Zeferino Vaz, 13083--859, Campinas, SP, Brazil}
\email{ddnovaes@unicamp.br$^1$}
\email{pedro.pereira@ime.unicamp.br$^2$}



\begin{abstract}
Floquet's Theorem is a celebrated result in the theory of ordinary differential equations. Essentially, the theorem states that, when studying a linear differential system with $T$-periodic coefficients, we can apply a, possibly complex, $T$-periodic change of variables that transforms it into a linear system with constant coefficients. In this paper, we explore further the question of the nature of this change of variables. We state necessary and sufficient conditions for it to be real and $T$-periodic. Failing those conditions, we prove that we can still find a real change of variables that is ``partially" $T$-periodic and ``partially" $T$-antiperiodic. We also present applications of this new form of Floquet's Theorem to the study of the behavior of solutions of nonlinear differential systems near periodic orbits.
\end{abstract}
\maketitle


\section{Introduction and statement of results}
Floquet's Theorem is a celebrated result in the theory of ordinary differential equations. After its original appearance in 1883 (see \cite{floquet}), it has been restated in a number of important textbooks on the subject (see, for instance, \cite{chiconeODE,coddingtonLevinsonODE,haleordinary}). Furthermore, many fruitful applications have been found for it, such as the study of quantum systems governed by a time-periodic Hamiltonian (see, for example, \cite{GRIFONI1998229,Holthaus_2016,shirley}), the computation of numerical approximations of local stable and unstable manifolds of periodic orbits  (see, for example, \cite{castelli}), and the obtainment of integral manifolds for perturbed differential equations via averaging theory (see, for example,  \cite{haleinvariant,NovPer2023}).

Essentially, the theorem states that, when studying a linear differential system with $T$-periodic coefficients, we can apply a $T$-periodic change of variables that transforms it into a linear system with constant coefficients. There is, however, a caveat, which can be quite bothersome in some cases: this change of variables possibly involves complex matrices. More precisely, consider a linear system of the form
\begin{equation}\label{systemmain}
	\dot \bx = A(t) \bx, \qquad \bx \in \R^n,
\end{equation}
where $A: \R \to \mathbb{R}^{n \times n}$ is $T$-periodic, $T>0$. Floquet's Theorem can be stated as follows:

\begin{theorem}[Floquet's Theorem] \label{theoremfloquet}
	Let $\Phi(t)$ be a fundamental matrix solution of \eqref{systemmain}. Then, there are $B \in \C^{n \times n}$ and a $T$-periodic matrix function $P:\R \to \C^{n \times n}$ such that $\Phi(t) = P(t) e^{tB}$ (the Floquet Normal Form of $\Phi(t)\,$). In particular, the change of variables $\by = P^{-1}(t) \bx$ transforms \eqref{systemmain} into $\dot \by = B \by$.
\end{theorem}
 
 It is well known, in fact, that a real periodic change of variables exists provided that we allow it to be $2T$-periodic, but in some applications this doubled period can be somewhat limiting (for instance, in \cite[Lemma 4.1]{haleinvariant}, the use of the complex form of Floquet's Theorem, so as to avoid doubling the period, imposes the need of verification of the fairly restrictive condition that the determinant of a matrix function is non-vanishing). In this paper, we explore further the question of the nature of the transformation of variables provided in Floquet's Theorem. 
 
 In order to effectively state the results obtained, we must define an important quantity associated to system \eqref{systemmain}. To ensure that this quantity is well-defined, we remind the reader that, if $\Phi(t)$ is a fundamental matrix solution of \eqref{systemmain}, then $\Phi^{-1}(0) \Phi(T)$ is called a \textbf{monodromy matrix} of this system, and that the Jordan normal form of this matrix does not depend on the choice of $\Phi(t)$. In fact, if $\Psi(t)$ and $\Phi(t)$ are two different fundamental matrix solutions of \eqref{systemmain}, then the monodromy matrices associated to those solutions are similar under the transformation $\Phi^{-1}(0)\Psi(0)$. We are now ready to define the \textbf{A-index} of a system of the form given in \eqref{systemmain}.
 
 \begin{definition}
 	The antiperiodicity index (A-index, for short) of system \eqref{systemmain} is the sum of the dimensions of Jordan blocks associated to a real negative eigenvalue that appear an odd number of times (each one counted only once) in the Jordan normal form of a monodromy matrix $\Phi^{-1}(0) \Phi(T)$ of \eqref{systemmain}.
 \end{definition}
 
 Having defined the A-index, we will show that we can always find a real change of variables that is ``partially" $T$-periodic and ``partially" $T$-antiperiodic transforming \eqref{systemmain} into a linear system with constant coefficients. Moreover, the A-index is a measure of the "degree of antiperiodicity" of this change of variables. More precisely, we will prove the following result:
\begin{mtheorem}\label{theoremmain}
Let $d \in \N$ be the A-index of \eqref{systemmain}. Then, there are $R \in \R^{n \times n}$, and $Q: \R \to \R^{n \times n}$ such that $\Phi(t):=Q(t)e^{tR}$ is a fundamental matrix solution of \eqref{systemmain} and $Q$ satisfies
	\[
	Q(t+T) = Q(t) \, \left[I_{n-d} \oplus (-I_d)\right]
	\]
	for all $t \in \R$, where the symbol $\oplus$ denotes the direct sum of square matrices. In particular, the change of variables $\by=Q^{-1}(t)\bx$ transforms \eqref{systemmain} into the real linear system with constant coefficients $\dot \by = R \by$.
\end{mtheorem}

The ideas presented above can be applied in the study of the behavior of solutions of perturbed nonlinear differential systems 
\begin{equation}\label{systempert}
	\dot \bz = f(\bz) +  g(t,\bz), \qquad \bz \in D \subset \R^n,\quad t \in I\subset \R,
\end{equation}
near non-trivial periodic orbits of the unperturbed one $	\dot \bz = f(\bz)$,
where $f$ and $g$ are of class $C^r$, $r\geq 2$.

If we suppose that $\dot \bz = f(\bz)$ admits a hyperbolic limit cycle, it is generally very useful to find a simplifying set of coordinates for \eqref{systempert} in a neighborhood of this cycle. This is the question addressed in the next result. In order to state this result, we remind the reader of the definition of the first variational equation of $\dot \bz = f(\bz)$ along $\varphi(t)$:
\begin{equation}\label{eq:variational}
	\dot \by = Df(\varphi(t))\, \by.
\end{equation}
Observe that the first variational equation is indeed a linear differential equation with periodic coefficients, establishing the link between the results herein presented. 
\begin{mtheorem} \label{theoremorbithyperbolic}
	Suppose that system $\dot \bz = f(\bz)$ admits a non-trivial $T$-periodic hyperbolic limit cycle $\varphi(t)$. Let $d \in \N$ be the A-index of the first variational equation \eqref{eq:variational}. Then, there is a real matrix function $U:\R \to \R^{(n-1) \times(n-1)}$ satisfying
	\[
	U(t+T) = U(t) \, \left[I_{n-d-1} \oplus (-I_d)\right]
	\]
	such that the change of variables $\bz \to (s,{\bf v}, {\bf w}) \in \R \times \R^{(n-d-1)\times (n-d-1)} \times \R^{d \times d}$ given by
	\[
	\bz = \varphi(s) + U(s) \left[\begin{array}{c}
		{\bf v} \\
		{\bf w}
	\end{array}\right]
	\] 
	transforms \eqref{systempert} into the system
	\begin{equation} \label{eq:systemfinalformhyperbolic}
		\begin{aligned}
			&\dot s = 1 + \Lambda_0(s,{\bf v},{\bf w}) +  \tilde{\Lambda}_0(t,s,{\bf v},{\bf w}), \\
			&\dot{\bf v} =  H_1\cdot {\bf v} +   \Lambda_1(s,{\bf v},{\bf w}) +  \tilde{\Lambda}_1(t,s,{\bf v},{\bf w}), \\
			&\dot {\bf w} =  H_2\cdot {\bf w} +  \Lambda_2(s,{\bf v},{\bf w}) +  \tilde{\Lambda}_2(t,s,{\bf v},{\bf w}),
		\end{aligned}
	\end{equation}
	where $H_1 \in \R^{(n-d-1) \times (n-d-1)}$, $H_2 \in \R^{d \times d}$, and the functions appearing in \eqref{eq:systemfinalformhyperbolic} satisfy the following properties:
	\begin{enumerate}[label=(\roman*)]
		\item \label{propertyeigenvalues} If $\{\lambda_\ell \in \C: \ell=1,2,\ldots,n-1 \}$ is the family of eigenvalues of the derivative of the Poincaré map associated to the periodic orbit $\varphi(t)$, then the real parts of the eigenvalues of $H_1 \oplus H_2$ are given by the family $\{ \frac{1}{T} \log |\lambda_\ell|: \ell=1,2,\ldots,n-1\}$.
		\item \label{propertylambda}$\Lambda_\ell(s,0,0)$ and $\frac{\partial \Lambda_\ell}{\partial ({\bf v},{\bf w})}(s,0,0)$ vanish for all $s \in \R$ and each $\ell \in \{0,1,2\}$.
		\item \label{periodiclambda0} $\Lambda_\ell(s+T,{\bf v}, {\bf w}) = \Lambda_\ell(s,{\bf v}, -{\bf w})$ and $\tilde{\Lambda}_\ell(t,s+T,{\bf v},{\bf w}) =\tilde{\Lambda}_\ell(t,s,{\bf v},-{\bf w})$ for $\ell \in \{0,1\}$.
		\item \label{periodiclambda2} $\Lambda_2(s+T,{\bf v}, {\bf w}) = -\Lambda_2(s,{\bf v}, -{\bf w})$ and $\tilde{\Lambda}_2(t,s+T,{\bf v},{\bf w}) =-\tilde{\Lambda}_2(t,s,{\bf v},-{\bf w})$.
		\item \label{propertylambdatilde} There is $C>0$ such that \[\left\|\tilde{\Lambda}_\ell(t,s,{\bf v},{\bf w})\right\| <C \left\| g\left(t,\,\varphi(s)+U(s) \left[\begin{array}{c}
			{\bf v} \\
			{\bf w}
		\end{array}\right] \right) \right\|\] for each $\ell \in \{0,1,2\}$.
	\end{enumerate}
	\end{mtheorem}

Theorem \ref{theoremorbithyperbolic} admits a generalization for the case when $\varphi(t)$ is not hyperbolic. In this case, the equation for $\dot s$ may require a linear term depending on ${\bf v}$.
\begin{mtheorem} \label{theoremorbit}
	Suppose that system $\dot \bz = f(\bz)$ admits a non-trivial $T$-periodic orbit $\varphi(t)$ and let $\Psi(t)$ be the principal fundamental matrix solution of the first variational equation \eqref{eq:variational}. Let $d \in \N$ be the A-index of \eqref{eq:variational} and $q_0$ be the largest $q \in \N$ such that the equation $(\Psi(T)-I_{n})^q \, v=\varphi'(0)$ admits a solution $v \in \R^n$. Then, $q_0\leq n-d-1$ and there is a real matrix function $U:\R \to \R^{n \times(n-1)}$ satisfying
	\[
	 U(t+T) = U(t) \, \left[I_{n-d-1} \oplus (-I_d)\right]
	\]
	such that the change of variables $\bz \to (s,{\bf v}, {\bf w}) \in \R \times \R^{(n-d-1)\times (n-d-1)} \times \R^{d \times d}$ given by
	\[
		\bz = \varphi(s) + U(s) \left[\begin{array}{c}
			{\bf v} \\
			{\bf w}
		\end{array}\right]
	\] 
	transforms \eqref{systempert} into the system
	\begin{equation} \label{eq:systemfinalform}
		\begin{aligned}
			&\dot s = 1 + L  \cdot {\bf v} + \Lambda_0(s,{\bf v},{\bf w}) +  \tilde{\Lambda}_0(t,s,{\bf v},{\bf w}), \\
			&\dot{\bf v} =  H_1\cdot {\bf v} +   \Lambda_1(s,{\bf v},{\bf w}) +  \tilde{\Lambda}_1(t,s,{\bf v},{\bf w}), \\
			&\dot {\bf w} =  H_2\cdot {\bf w} +  \Lambda_2(s,{\bf v},{\bf w}) +  \tilde{\Lambda}_2(t,s,{\bf v},{\bf w}),
		\end{aligned}
	\end{equation}
	where  $H_1 \in \R^{(n-d-1) \times (n-d-1)}$, $H_2 \in \R^{d \times d}$, and $L \in \R^{1 \times (n-d-1)}$. The functions appearing in \eqref{eq:systemfinalform} satisfy the same properties given in Theorem \ref{theoremorbithyperbolic}, as well as the additional property:
	\begin{enumerate}[label=(\roman*)]
		\setcounter{enumi}{5}
		\item If $q_0=0$, then $L=0$. If $1\leq q_0\leq n-d-1$, then $L=\left[L_1 \, L_2 \, \cdots L_{q_0} \, 0 \cdots 0\right]$, where $L_\ell \in \R$ for each $\ell \in \{1,2,\ldots,q_0\}$.
	\end{enumerate}
	\end{mtheorem}

\section{Preliminary results}
Before presenting the proofs of the results stated in the Introduction, we state some preliminary results that will be needed throughout the remainder of the article. Henceforth, we let $I_n$ denote the $n \times n$ identity matrix.
First, we prove a generalized version of Floquet's Theorem that is more suited to our study.

\begin{theorem}[General Floquet Normal Form] \label{theoremfloquetrepresentation}
	Let $\Phi(t)$ be a fundamental matrix solution of \eqref{systemmain} and $k \in \N$. For each $B \in \mathbb{C}^{n \times n}$ satisfying 
	\begin{equation}\label{eq:floquethypothesis}
		e^{kTB} = (\Phi^{-1}(0) \Phi(T))^k,
	\end{equation}
	there is a $kT$-periodic matrix function $P:\R \to \C^{n \times n}$ such that $\Phi(t) = P(t) e^{tB}$. In particular, the change of variables $\by = P^{-1}(t) \bx$ transforms \eqref{systemmain} into $\dot \by = B \by$.
\end{theorem}
\begin{proof}
	Define $C:=  \Phi^{-1}(0) \Phi(T)$ and observe that $t\mapsto \Phi(t) C$ and $t \mapsto \Phi(t+T)$ are solutions of the same initial value problem, so that $\Phi(t+T) = \Phi(t) C$. Hence, for each $k \in \N$, 
	\[
		\Phi(t+kT) = \Phi(t) C^k.
	\]
	Let $B \in \C^{n \times n}$ satisfy \eqref{eq:floquethypothesis}. The proof is concluded by defining $P(t):=\Phi(t) e^{-tB}$, since it is straightforward to verify that $P$ is $kT$-periodic.
\end{proof}
Observe that we can retrieve Theorem \ref{theoremfloquet} from Theorem \ref{theoremfloquetrepresentation} by taking $k=1$ and considering that every invertible complex matrix has a, possibly complex, matrix logarithm.

We now state a known result concerning the existence of a real logarithm of a matrix (for a proof, see \cite{culver}).
\begin{lemma} \label{lemmareallog}
	Let $C$ be a square real matrix. There exists a real solution $X$ to the equation $C=e^X$ if, and only if, $C$ is nonsingular and each Jordan block of the Jordan normal form of $C$ associated to a real negative eigenvalue occurs an even number of times.
\end{lemma}

Finally, we present a result concerning the logarithm of a matrix in the Jordan normal form (for a proof, see \cite[Chapter VIII, Section 8]{gantmacher}).
\begin{lemma} \label{lemmalogmatrix}
	Let $J=J_1 \oplus \cdots \oplus J_k$ be a matrix in the Jordan normal form, where $J_\ell$ is a Jordan block associated to the eigenvalue $\lambda_\ell$, and, for each $\lambda \in \C$, let $\log \lambda$ denote any of the possible values $\xi \in \C$ satisfying $e^\xi = \lambda$. Then, a solution to the matrix equation $e^X=J$ is given by $X= \log J := \log J_1 \oplus \cdots \oplus \log J_k$, where, for each $(m+1)$-dimensional Jordan block $J_\ell$, $1\leq \ell \leq k$, we let $N$ be the $(m+1)$-dimensional nilpotent square matrix having the element 1 in each position of the superdiagonal, that is, $N=(\delta_{i+1,j})_{1\leq i,j\leq m+1}$, and define
	\[
		\log J_\ell := I_{m+1} \log \lambda_\ell  + N \frac{1}{\lambda_\ell} + N^2 \left(\frac{-1}{2 \lambda_\ell^2}\right) +\cdots + N^m \frac{(-1)^{m-1}}{m \lambda^m}.
	\]
\end{lemma}
\section{Proof of Theorem \ref{theoremmain}} \label{sec:proofoftheoremmain}
Let $\Psi(t)$ be a fundamental matrix solution of $\eqref{systemmain}$. Then, for each nonsingular matrix $S \in \R^{n \times n}$, the product $\Psi(t) S$ is also a fundamental matrix solution of $\eqref{systemmain}$.

Let $S \in \R^{n \times n}$ be such that $S^{-1} \Psi^{-1}(0) \Psi(T) S$ is in the real Jordan canonical form. By changing order of blocks in the canonical form, $S$ can be chosen so that
\[
S^{-1} \Psi^{-1}(0) \Psi(T) S = \mathcal{J}_1 \oplus \mathcal{J}_2,
\]
where $\mathcal{J}_1 \in \R^{(n-d) \times (n-d)}$ and $\mathcal{J}_2 \in \R^{d \times d}$ are matrices in the real Jordan canonical form satisfying the following condition: each Jordan block of $\mathcal{J}_1$ associated to a real negative eigenvalue occurs an even number of times, and every Jordan block of $\mathcal{J}_2$ is associated to a real negative eigenvalue and occurs only once in this matrix.
Then, the fundamental matrix solution $\Phi(t):= \Psi(t) S$ satisfies
\begin{equation}
	\Phi^{-1}(0) \Phi(T) = \mathcal{J}_1 \oplus \mathcal{J}_2.
\end{equation}

By Lemma \ref{lemmareallog}, there is a real matrix $R_1 \in \R^{(n-d) \times (n-d)}$ such that $e^{TR_1} = \mathcal{J}_1$. Furthemore, by considering Lemma \ref{lemmalogmatrix} and the fact that every eigenvalue of $\mathcal{J}_2$ is negative and real, it follows that there exists $R_2 \in \R^{d \times d}$ such that $e^{T R_2 + i\pi I_d}= \mathcal{J}_2$.

Define the matrices
\[\tilde{R}:= R_1 \oplus \left(R_2+i\frac{\pi}{T}I_d\right) \qquad \text{and} \qquad R:= R_1 \oplus R_2. \]
It is easy to see that $e^{T \tilde{R}} = \Phi^{-1}(0) \Phi(T)$ and $e^{2T R} = e^{2T\tilde{R}} =  (\Phi^{-1}(0) \Phi(T))^2$. Thus, Theorem \ref{theoremfloquetrepresentation} ensures that there are a $T$-periodic matrix function $t \mapsto P(t) \in \mathbb{C}^{n \times n}$ and a $2T$-periodic matrix function $t\mapsto Q(t) \in \mathbb{C}^{n \times n}$ such that 
\[\Phi(t)= P(t) e^{t\tilde{R}} = Q(t) e^{tR}.\]
Observe, since $\Phi(t) e^{-tR} \in \R^{n \times n}$, it follows that $Q(t) \in \R^{n \times n}$. Considering also that $R$ and $\tilde{R}$ clearly commute, it follows that \[Q(t+T) = P(t+T) e^{(t+T)\tilde{R}} e^{-(t+T)R} = \Phi(t) e^{-tR} e^{T(\tilde{R}-R)} = Q(t) \left[I_{n-d} \oplus (-I_d)\right],\] 
for all $t \in \R$. This concludes the proof.

\section{Proofs of Theorems \ref{theoremorbithyperbolic} and \ref{theoremorbit}}
We will prove the more general Theorem \ref{theoremorbit}, since the only difference in the proof of Theorem \ref{theoremorbithyperbolic} is that the hyperbolicity of the periodic orbit $\varphi(t)$ ensures that the matrix $L$ appearing in \eqref{eq:systemfinalform} vanishes.

Let $d$ and $q_0$ be as in the statements of Theorems \ref{theoremorbithyperbolic} and \ref{theoremorbit}. Also, let $\Psi(t)$ be the principal fundamental matrix solution of \eqref{eq:variational}. Then, $\varphi'(0)$ is an eigenvector of the monodromy matrix $\Psi^{-1}(0) \Psi(T)=\Psi(T)$ associated to the eigenvalue $1$. Let $\beta=\{b_1,\ldots,b_n\}$ be a real Jordan basis for the operator $\Psi(T)$ such that $b_1=\varphi'(0)$. Hence, if $S$ is the matrix whose $k$-th column is given by $b_k$, then $$S^{-1} \Psi^{-1}(0) \Psi(T) S = \mathcal{J}_\varphi \oplus \mathcal{J}, $$
where $\mathcal{J}_\varphi \in \R^{n_0 \times n_0}$ is a Jordan block with eigenvalue $1$ and $\mathcal{J}$ is in the real Jordan normal form. If $\varphi(t)$ is hyperbolic, $\mathcal{J}_\varphi$ is the matrix with a single element given by $1$.

Define the fundamental matrix solution $\Phi(t)$ of \eqref{eq:variational} by $\Phi(t):=\Psi(t) S$. In this case, $\Phi^{-1}(0)\Phi(T)= \mathcal{J}_\varphi \oplus \mathcal{J}$. Let $\{1,\lambda_1,\ldots,\lambda_{n-1}\}$ be the eigenvalues (counting multiplicities) of the monodromy matrix $\Phi^{-1}(0)\Phi(T)$.

Observe that, since $S {\bf e}_{\ell} = b_{\ell}$ for all $\ell \in \{1,\ldots,n\}$, we have that $$\left[S^{-1}( \Psi^{-1}(0) \Psi(T)  - I_n)S\right]^{n_0-1} {\bf e}_{n_0}= S^{-1} (\Psi(T) - I_n)^{n_0-1} \, b_{n_0} = {\bf e}_1.$$
Thus, it follows that $(\Psi(T) - I_n)^{n_0-1} \, b_{n_0}=b_1=\varphi'(0)$, so that $q_0\geq n_0-1$. Moreover, since $\beta$ is a Jordan basis, $n_0-1$ is the largest $k \in \N$ for which $(\Psi(T) - I_n)^k \, v = b_1$ has a solution $v \in \R^n$. Hence, $n_0=q_0+1$. 

Choose the vectors $b_{n_0+1},\ldots,b_n$ of $\beta$ such that
$$ \mathcal{J} = \mathcal{J}_1 \oplus \mathcal{J}_2,$$ with $\mathcal{J}_1 \in \R^{n_1 \times n_1}$ and $\mathcal{J}_2 \in \R^{n_2 \times n_2}$ in the real Jordan canonical form satisfying the following condition: each Jordan block of $\mathcal{J}_1$ associated to a real negative eigenvalue occurs an even number of times, and every Jordan block of $\mathcal{J}_2$ is associated to a real negative eigenvalue and occurs only once. It is then clear that $n_2=d$. Thus, since $n_0+ n_1+n_2=n$, it follows that $n_1=n-d-q_0-1$.

By Lemma \ref{lemmareallog}, there are $R_0 \in \R^{n_0 \times n_0}$ and $R_1 \in \R^{n_1 \times n_1}$ such that $e^{TR_0}=\mathcal{J}_0$ and $e^{TR_1} = \mathcal{J}_1$. Furthemore, by considering Lemma \ref{lemmalogmatrix} and the fact that every eigenvalue of $\mathcal{J}_2$ is negative and real, it follows that there exists $R_2 \in \R^{n_2 \times n_2}$ such that $e^{T R_2 + i\pi I_d}= \mathcal{J}_2$.

Define the matrices
\[\tilde{R}:= R_0 \oplus R_1 \oplus \left(R_2+i\frac{\pi}{T}I_{n_2}\right) \qquad \text{and} \qquad R:= R_0 \oplus R_1 \oplus R_2. \]
Let $\{\mu_0,\ldots,\mu_{n-1}\}$ be the eigenvalues (counting multiplicities) of $R$. Since it is clear that $e^{2TR} = (\Phi^{-1}(0) \Phi(T))^2$, it follows that $e^{2T\mu_0} = 1$ and $e^{2T\mu_k} = \lambda_k^2$ for $k=1,\ldots,n-1$. Therefore, $\mu_0=0$ and 
\begin{equation}\label{eq:eigenvalues}
	\text{Re}(\mu_k) = \frac{1}{2T} \text{Re} \left(\log \lambda_k^2\right) = \frac{\log |\lambda_k|}{T}
\end{equation}
for $k =1,\ldots,n-1$.

Proceeding exactly as in the proof of Theorem \ref{theoremmain}, we obtain $Q:\R \to \R^{n \times n}$ such that $\Phi(t)=Q(t)e^{TR}$ and \[Q(t+T) = Q(t) \left[I_{n-d} \oplus (-I_{d})\right],\] 
for all $t \in \R$. Observe that the first column $c_1(t)$ of $\Phi(t)=\Psi(t) S$ is a solution of \eqref{eq:variational} satisfying $c_1(0)=\varphi'(0)$. Therefore, since $\varphi'(t)$ is a solution of \eqref{eq:variational}, it follows that $c_1(t)=\varphi'(t)$. Since the first column of $R$ is zero, we conclude that the first column of $Q(t)=\Phi(t) e^{-tR}$ is given by $\varphi'(t)$.

Observe that $R$ is of the form 
\begin{equation}\label{eq9}
	R = \left[\begin{array}{c|ccc}
		0 &   \multicolumn{3}{c}{L_{*}}\\
		\hline \\[-1.5\normalbaselineskip]
		0 &  &  &  \\
		\vdots &\multicolumn{3}{c}{\smash{\raisebox{.0\normalbaselineskip}{$\text{\huge$ H $}$}}}  \\
		0 &  &  & 
	\end{array}\right],
\end{equation}
where $L_{*} \in \R^{1 \times (n-1)}$ vanishes if $q_0=0$ and there are $L_1,\ldots,L_{q_0} \in \R$ such that $L_{*}=[L_1 \ldots L_{q_0} \, 0 \ldots 0]$ if $q_0\geq1$. Moreover, if $R_0^{1,1}$ is the submatrix of $R_0$ obtained by removing its first row and its first column, then $H=H_1\oplus H_2$, where $H_1:=R_0^{1,1} \oplus R_1$ and $H_2:=R_2$. As we will prove below, $H_1$ and $H_2$ are the matrices appearing in the transformed system presented in the statement of the Theorem.  Then, considering that \eqref{eq9} ensures that the characteristic polynomials $p_R$ and $p_H$ of $R$ and $H$ satisfy the identity $p_R(x) = x p_H(x)$, property (i) follows from \eqref{eq:eigenvalues}.

since $H$ is upper-triangular and its diagonal elements are, with the exception of one count of zero, the eigenvalues (counting multiplicities) of $R$, property \ref{propertyeigenvalues} follows from \eqref{eq:eigenvalues}.

Let $U(t)$ be the $n \times (n-1)$ matrix whose columns are given by the last $n-1$ columns of $Q(t)$. It is then clear that 
\begin{equation}\label{eq:Uperiod}
	U(t+T) = U(t) \, A_d,
\end{equation}
 where $A_d:= I_{n-d-1} \oplus (-I_d)$. Moreover, since $\Phi(t)$ is a solution of \eqref{eq:variational}, it follows that $Q'(t)+ Q(t)R = Df(\varphi(t))Q(t)$. Restricting this equality to the last $n-1$ columns, we obtain
\begin{equation}\label{eq:identityU}
	U'(t) + \varphi'(t) L_* + U(t) H = Df(\varphi(t)) U(t)
\end{equation}
for all $t \in \R$.

Consider the transformation of variables 
\[
\bz = \varphi(s) + U(s) \left[\begin{array}{c}
	{\bf v} \\
	{\bf w}
\end{array}\right] =\varphi(s) + U(s) {\bf h},
\]
where ${\bf h}:= ({\bf v},{\bf w})$. For convenience, we will adopt the notation ${\bf h} = ({\bf v},{\bf w})$ for the rest of the proof. Observe that, under this transformation, \eqref{systempert} becomes
\begin{equation} \label{eq:transformation1}
	\left[\varphi'(s) + U'(s) {\bf h} \;\vline \; U(s) \right] \left[\begin{array}{c}
		\dot s \\
		\dot {\bf h}
	\end{array}\right] = f(\varphi(s)+U(s){\bf h}) + g(t,\varphi(s)+U(s){\bf h}).
\end{equation}
Since the coefficient matrix appearing in left-hand side becomes the invertible matrix $Q(s)$ when ${\bf h}=0$, it follows that, for $\|{\bf h}\|$ small, \eqref{eq:transformation1} can be solved for $(\dot s , \,\dot {\bf h})$. 

We set out to prove properties \ref{propertylambda} to \ref{propertylambdatilde}. In order to do so, we must find explicit equations for $\dot s$ and $\dot {\bf h}$. Let $\eta(s) \in \R^n$ be the unitary vector that is orthogonal to each of the $n-1$ columns of the matrix $U(s)$. Then, if $\eta(s)^T$ denotes the transpose of the column vector $\eta(s)$, it follows that $\eta(s)^T \cdot U(s) = 0$ and $\eta(s)^T \cdot \eta(s) = 1$ for all $s \in \R$. Observe that, since $U$ is of class $C^r$, so is $\eta$. Moreover, since $U$ satisfies \eqref{eq:Uperiod}, it follows that $\eta(s)$ is either $T$-periodic or $T$-antiperiodic. Then, by projecting \eqref{eq:transformation1} onto $\eta(s)$, we obtain
\begin{equation}\label{eq:sdot}
	\dot s = \frac{\eta(s)^T \cdot f(\varphi(s) +U(s) {\bf h})}{\eta(s)^T \cdot (\varphi'(s) + U'(s) {\bf h})} + \frac{\eta(s)^T \cdot g(t,\varphi(s) +U(s) {\bf h})}{\eta(s)^T \cdot (\varphi'(s) + U'(s) {\bf h})}.
\end{equation}

Having found an explicit form of the equation for $\dot s$, we can prove the properties concerning functions $\Lambda_0$ and $\tilde{\Lambda}_0$ appearing in the statement of the Theorem. Let $F_0(s,{\bf h})$ and $\tilde{\Lambda}_0(t,s,{\bf h})$ be defined by
\[
	F_0(s,{\bf h}): = \frac{\eta(s)^T \cdot f(\varphi(s) +U(s) {\bf h})}{\eta(s)^T \cdot (\varphi'(s) + U'(s) {\bf h})}, \qquad \tilde{\Lambda}_0(t,s,{\bf h}):= \frac{\eta(s)^T \cdot g(t,\varphi(s) +U(s) {\bf h})}{\eta(s)^T \cdot (\varphi'(s) + U'(s) {\bf h})}.
\]
Then, it is clear that property \ref{propertylambdatilde} holds for $\ell =0$, because $\|{\bf h}\|$ is small and $\eta(s)^T \cdot \varphi'(s)>0$ is $2T$-periodic. It is also easy to see that $F_0(s,0)= 1$. Moreover, by considering \eqref{eq:identityU}, it follows that
\[
	\frac{\partial F_0}{\partial {\bf h}}(s,0) = \frac{\eta(s)^T\cdot \left(Df(\varphi(s))U(s)-U'(s)\right)}{\eta^T(s) \cdot \varphi'(s)} = L_*.
\] 
By defining $\Lambda_0(s,{\bf h}):= F_0(s,{\bf h})-1-L_*{\bf h}$, we conclude that \eqref{eq:sdot} becomes $$\dot s = 1+L_*{\bf h} + \Lambda_0(s,{\bf h}) + \tilde{\Lambda}_0(t,s,{\bf h}).$$ Since $q_0 = n_0-1 \leq n_0+n_1-1 = n-d-1$, it follows that $L_*{\bf h} = L{\bf v}$, where $L \in \R^{1 \times (n-d-1)}$ is such that $L_*=[L \,|\, 0 \cdots 0]$. It is clear from the definition of $\Lambda_0$ that property \ref{propertylambda} holds for $\ell=0$. Moreover, property \ref{periodiclambda0} for $\ell=0$ is a straightforward consequence of the identity $L_* {\bf h} = L_* A_d {\bf h}$ combined with the fact that $F_0(s+T,{\bf h}) = F_0(s,A_d {\bf h})$, because $\eta(s)$ is either $T$-periodic or $T$-antiperiodic.

Let $U_k(s)$ be the $k$-th column of $U(s)$. For each $j \in \{1,\ldots,n-1\}$, we proceed as when choosing $\eta(s)$ and find $\xi_j(s) \in \R^n$ such that $\xi_j(s)^T \cdot U_k(s) = \delta_{kj}$ and $\xi_j(s)^T \cdot \varphi'(s)=0$, where $\delta_{kj}$ is the Kronecker delta. Once again, it is clear that each $\xi_j(s)$ is of class $C^r$. Define $\xi(s) \in \R^{n \times (n-1)}$ as the matrix whose columns are given by $\xi_j(s)$. It is then easy to verify that $$\xi(s+T)= \xi(s)A_d.$$ By projecting \eqref{eq:transformation1} onto $\xi(s)$, we obtain:
\begin{equation}\label{eq:hdot}
	\dot {\bf h} = G(s,{\bf h}) + \tilde{G}(t,s,{\bf h}),
\end{equation}
where 
\[
	G(s,{\bf h}) := \xi(s)^T \Big(f\big(\varphi(s)+U(s){\bf h}\big) - U'(s)\,{\bf h}\big(1+L_*{\bf h}+\Lambda_0(s,{\bf h})\big)\Big)
\]
and
\[
	 \tilde{G}(t,s,{\bf h}) := \xi(s)^T\Big(g(t,\varphi(s)+U(s){\bf h})  - U'(s) {\bf h}\, \tilde{\Lambda}_0(t,s,{\bf h})\Big).
\]

We can now prove the properties stated in the Theorem concerning $\Lambda_1$, $\Lambda_2$, $\tilde{\Lambda}_1$, and $\tilde{\Lambda}_2$. Observe that $G(s,0)=0$, because $f(\varphi(s))=\varphi'(s)$. Moreover, 
\[
	\frac{\partial G}{\partial {\bf h}}(s,0)= \xi(s)^T \left(Df(\varphi(s)) U(s) - U'(s)\right) = \xi(s)^T \left(\varphi'(s) L_* + U(s) H\right) = H.
\]
By recalling that ${\bf h} = ({\bf v}, {\bf w})$ and $H=H_1 \oplus H_2$, we define 
\[
	\left[\begin{array}{c}
		\Lambda_1(s,{\bf h}) \\
		\Lambda_2(s,{\bf h})
	\end{array}\right] := G(s,{\bf h}) - H{\bf h} , \quad \left[\begin{array}{c}
\tilde{\Lambda}_1(t,s,{\bf h}) \\
\tilde{\Lambda}_2(t,s,{\bf h})
\end{array}\right] := \tilde{G}(t,s,{\bf h}),
\]
and it follows that \eqref{eq:hdot} becomes
\[	
	\begin{aligned}
	&\dot{\bf v} =  H_1\cdot {\bf v} +   \Lambda_1(s,{\bf v},{\bf w}) +  \tilde{\Lambda}_1(t,s,{\bf v},{\bf w}), \\
	&\dot {\bf w} =  H_2\cdot {\bf w} +  \Lambda_2(s,{\bf v},{\bf w}) +  \tilde{\Lambda}_2(t,s,{\bf v},{\bf w}).
	\end{aligned}
\]
Considering that $G(s,0)=0$ and $\frac{\partial G}{\partial {\bf h}}(s,0)=H$, it is easy to see that property \ref{propertylambda} holds for $\ell=1$ and $\ell=2$. Furthermore, since $\xi(s)^T$ is $2T$-periodic, it follows from the definition of $\tilde{G}$ that property \ref{propertylambdatilde} holds for $\ell=1$ and $\ell=2$.

Since $L_*{\bf h} = L_* A_d {\bf h}$ for every ${\bf h} \in \R^{n-1}$, and considering that $\xi(s+T)= \xi(s)A_d$, we conclude that $$G(s+T,{\bf h}) = A_dG(s,A_d{\bf h}).$$ Also, since $H=H_1\oplus H_2$, where $H_1 \in \R^{(n-d-1) \times (n-d-1)}$ and $H_2 \in \R^{d \times d}$, it is easy to see that $H$ and $A_d$ commute. Hence, $H{\bf h} = A_d H A_d {\bf h}$. Thus, it follows that
\[
\left[\begin{array}{c}
	\Lambda_1(s+T,{\bf h}) \\
	\Lambda_2(s+T,{\bf h})
\end{array}\right] := G(s+T,{\bf h}) - H{\bf h} = A_d \left(G(s,A_d{\bf h}) - H A_d {\bf h}\right) = A_d \left[\begin{array}{c}
\Lambda_1(s,A_d{\bf h}) \\
\Lambda_2(s,A_d{\bf h})
\end{array}\right].
\]
It is then clear that $\Lambda_1$ and $\Lambda_2$ satisfy properties \ref{periodiclambda0} and \ref{periodiclambda2}. Considering also that $$\tilde{G}(t,s+T,{\bf h}) = A_d\tilde{G}(t,s,A_d{\bf h}),$$ we conclude that $\tilde{\Lambda}_1$ and $\tilde{\Lambda}_2$ also satisfy those properties.

\section{Conditions for the existence of a real $T$-periodic Floquet normal form}
We provide necessary and sufficient conditions in terms of the A-index for the $T$-periodic transformations provided in Theorem \ref{theoremfloquet} to be of real nature. In fact, we prove that the existence of such transformation is equivalent to the vanishing of the A-index.
\begin{theorem}
	The following statements are equivalent:
	\begin{enumerate} [label=(\alph*)]
		\item \label{cnd2} The A-index of \eqref{systemmain} is zero.
		\item \label{cnd1} There are $R \in \R^{n \times n}$ and a $T$-periodic function $Q: \R \to \R^{n \times n}$ such that $\Psi(t):=Q(t)e^{tR}$ is a fundamental matrix solution of \eqref{systemmain}.
		\item \label{cnd3} Any fundamental matrix solution $\Phi(t)$ of \eqref{systemmain} admits a real Floquet normal form $\Phi(t) = P(t) e^{tB}$, where $B \in \R^{n \times n}$ and $P: \R \to \R^{n \times n}$ is $T$-periodic.
	\end{enumerate}
\end{theorem}
\begin{proof}
	It is a direct consequence of Theorem \ref{theoremmain} that \ref{cnd2} implies \ref{cnd1}.
	
Suppose that \ref{cnd1} holds for $\Psi(t)=Q(t)e^{tR}$ and let $\Phi(t)$ be any fundamental matrix solution of \eqref{systemmain}.	
\[
\Phi(t) = \Psi(t) \underbrace{\Psi^{-1}(0) \Phi(0)}_S= Q(t)S S^{-1}e^{tR}S= Q(t)S e^{tS^{-1} R S}=P(t) e^{t B},
\]
where $P(t)=Q(t)S$ and $B=S^{-1} R S$.
	
	Finally, suppose that \ref{cnd3} holds. Accordingly, let $\Phi(t)=P(t)e^{tB}$ be a fundamental matrix solution of \eqref{systemmain}, with $B \in \R^{n \times n}$. Then, since $e^{TB} = \Phi^{-1}(0) \Phi(T)$, it follows from Lemma \ref{lemmareallog} that each Jordan block of the Jordan normal form of $\Phi^{-1}(0) \Phi(T)$ associated to a real negative eigenvalue occurs an even number of times. Thus, the A-index of \eqref{systemmain} is clearly zero. This establishes that \ref{cnd3} implies \ref{cnd2}, concluding the proof of the Theorem.
\end{proof}

In particular, by reminding the reader that the eigenvalues of any monodromy matrix of \eqref{systemmain} are called the \textbf{characteristic multipliers} of this system, we obtain the following Corollary:
\begin{corollary}
	Let $\Phi(t)$ be a fundamental matrix of \eqref{systemmain}. If every real characteristic multiplier of \eqref{systemmain} is non-negative, then there are a real matrix $B \in \R^{n \times n}$ and a $T$-periodic real matrix function $P: \R \to \R^{n \times n}$ such that $\Phi(t)=P(t)e^{tB}$.
\end{corollary}

\section*{Conflict of interest}
On behalf of all authors, the corresponding author states that there is no conflict of interest.

\section*{Acknowledgements}

DDN is supported by S\~{a}o Paulo Research Foundation (FAPESP) grants 2022/09633-5, 2019/10269-3, and 2018/13481-0, and by Conselho Nacional de Desenvolvimento Cient\'{i}fico e Tecnol\'{o}gico (CNPq) grant 309110/2021-1. PCCRP is supported by S\~{a}o Paulo Research Foundation (FAPESP) grant 2020/14232-4.

\bibliographystyle{abbrv}
\bibliography{references}

\begin{thebibliography}{10}

\bibitem{castelli}
R.~Castelli, J.-P. Lessard, and J.~D. Mireles~James.
\newblock Parameterization of invariant manifolds for periodic orbits {I}:
  Efficient numerics via the {F}loquet normal form.
\newblock {\em SIAM Journal on Applied Dynamical Systems}, 14(1):132--167,
  2015.

\bibitem{chiconeODE}
C.~Chicone.
\newblock {\em Ordinary differential equations with applications}, volume~34 of
  {\em Texts in Applied Mathematics}.
\newblock Springer, New York, second edition, 2006.

\bibitem{coddingtonLevinsonODE}
E.~A. Coddington and N.~Levinson.
\newblock {\em Theory of ordinary differential equations}.
\newblock McGraw-Hill Book Co., Inc., New York-Toronto-London, 1955.

\bibitem{culver}
W.~J. Culver.
\newblock On the existence and uniqueness of the real logarithm of a matrix.
\newblock {\em Proc. Amer. Math. Soc.}, 17:1146--1151, 1966.

\bibitem{floquet}
G.~Floquet.
\newblock Sur les \'{e}quations diff\'{e}rentielles lin\'{e}aires \`a
  coefficients p\'{e}riodiques.
\newblock {\em Ann. Sci. \'{E}cole Norm. Sup. (2)}, 12:47--88, 1883.

\bibitem{gantmacher}
F.~R. Gantmacher.
\newblock {\em The theory of matrices. {V}ols. 1, 2}.
\newblock Chelsea Publishing Co., New York, 1959.
\newblock Translated by K. A. Hirsch.

\bibitem{GRIFONI1998229}
M.~Grifoni and P.~Hänggi.
\newblock Driven quantum tunneling.
\newblock {\em Physics Reports}, 304(5):229--354, 1998.

\bibitem{haleinvariant}
J.~K. Hale.
\newblock Integral manifolds of perturbed differential systems.
\newblock {\em Ann. of Math. (2)}, 73:496--531, 1961.

\bibitem{haleordinary}
J.~K. Hale.
\newblock {\em Ordinary differential equations}.
\newblock Robert E. Krieger Publishing Co., Inc., Huntington, N.Y., second
  edition, 1980.

\bibitem{Holthaus_2016}
M.~Holthaus.
\newblock Floquet engineering with quasienergy bands of periodically driven
  optical lattices.
\newblock {\em Journal of Physics B: Atomic, Molecular and Optical Physics},
  49(1):013001, nov 2015.

\bibitem{NovPer2023}
D.~D. Novaes and P.~C. C.~R. Pereira.
\newblock Invariant tori via higher order averaging method: existence,
  regularity, convergence, stability, and dynamics.
\newblock {\em Mathematische Annalen}, June 2023.

\bibitem{shirley}
J.~H. Shirley.
\newblock Solution of the {S}chr\"odinger equation with a {H}amiltonian
  periodic in time.
\newblock {\em Phys. Rev.}, 138:B979--B987, May 1965.

\end{thebibliography}

\end{document}